\documentclass[12pt]{article}
\usepackage{amsmath,amsthm,amssymb}
\newcommand{\remove}[1]{}
\newtheorem{thm}{Theorem}[section]
\newtheorem{claim}[thm]{Claim}
\newtheorem{lem}[thm]{Lemma}
\newtheorem{define}[thm]{Definition}
\newtheorem{cor}[thm]{Corollary}

\newtheorem{THM}{Theorem}

\def\F{{\mathbb{F}}}

\def\R{{\mathbb{R}}}

\def\cL{{\mathcal L}}

\def\_{\,\,\,\,\,}

\newcommand{\eps}{\epsilon}

\begin{document}

\title{On the size of Kakeya sets in finite fields}
\author{Zeev Dvir \thanks{Department of Computer Science, Weizmann
institute of science, Rehovot, Israel.  {\tt
zeev.dvir@weizmann.ac.il}. Research supported by Binational
Science Foundation (BSF) grant.}}
\date{}
\maketitle

\begin{abstract}
A Kakeya set is a subset of $\F^n$, where $\F$ is a finite field
of $q$ elements, that contains a line in every direction. In this
paper we show that the size of every Kakeya set is at least $C_{n}
\cdot q^{n}$, where $C_{n}$ depends only on $n$. This improves the
previously best lower bound for general $n$ of $\approx q^{4n/7}$.
\end{abstract}
\pagenumbering{arabic}

\section{Introduction}

Let $\F$ denote a finite field of $q$ elements. A {\sf Kakeya} set
(also called a {\sf Besicovitch} set) in $\F^n$ is a set $K
\subset \F^n$ such that $K$ contains a line in every direction.
More formally, $K$ is a Kakeya set if for every $x \in \F^n$ there
exists a point $y \in \F^n$ such that the line
\[ L_{y, x} \triangleq \{ y + a\cdot x | a \in \F \} \]
is contained in $K$.

The motivation for studying Kakeya sets over finite fields is to
try and understand better the more complicated questions regarding
Kakeya sets in $\R^n$. A Kakeya set $K \subset \R^n$ is a compact
set containing a line segment of unit length in every direction.
The famous Kakeya Conjecture states that such sets must have
Hausdorff (or Minkowski) dimension equal to $n$. The importance of
this conjecture is partially due to the connections it has to many
problems in harmonic analysis, number theory and PDE. This
conjecture was proved for $n=2$ \cite{Davies71} and is open for
larger values of $n$ (we refer the reader to the survey papers
\cite{Wolff99,Bou00,Tao01} for more information)

It was first suggested by Wolff \cite{Wolff99} to study finite
field Kakeya sets. It was asked in \cite{Wolff99} whether there
exists a lower bound of the form $C_n \cdot q^n$ on the size of
such sets in $\F^n$. The lower bound appearing in \cite{Wolff99}
was of the form $C_n \cdot q^{(n+2)/2}$. This bound was further
improved in \cite{Rogers01,BKT04,MT04,Tao08} both for general $n$
and for specific small values of $n$  (e.g for $n=3,4$). For
general $n$, the currently best lower bound is the one obtained in
\cite{Rogers01,MT04} (based on results from \cite{KT99}) of $C_n
\cdot q^{4n/7}$. The main technique used to show this bound is an
additive number theoretic lemma relating the sizes of different
sum sets of the form $A + r \cdot B$ where $A$ and $B$ are fixed
sets in $\F^n$ and $r$ ranges over several different values in
$\F$ (the idea to use additive number theory in the context of
Kakeya sets is due to Bourgain \cite{Bou99}).

The next theorem gives a near-optimal bound on the size of Kakeya
sets. Roughly speaking, the proof follows by observing that any
degree $q-2$ homogenous polynomial in $\F[x_1,\ldots,x_n]$ can be
`reconstructed' from its value on any Kakeya set $K \subset \F^n$.
This implies that the size of $K$ is at least the dimension of the
space of polynomials of degree $q-2$, which is $\approx q^{n-1}$
(when $q$ is large).
\begin{THM}\label{thm-Kakeya}
Let $K \subset \F^n$ be a Kakeya set. Then
\[ |K| \geq C_n \cdot q^{n-1}, \]
where $C_n$ depends only on $n$.
\end{THM}
The result of Theorem~\ref{thm-Kakeya} can be made into an even
better bound using the simple observation that a product of Kakeya
sets is also a Kakeya set.
\begin{cor}\label{cor-Kakeya}
For every integer $n$ and every $\eps> 0$ there exists a constant
$C_{n,\eps}$, depending only on $n$ and $\eps$ such that any
Kakeya set $K \subset \F^n$ satisfies
\[ |K| \geq C_{n,\eps} \cdot q^{n-\eps}, \]
\end{cor}
\begin{proof}
Observe that, for every integer $r>0$, the Cartesian product $K^r
\subset \F^{n\cdot r}$ is also a Kakeya set. Using
Theorem~\ref{thm-Kakeya} on this set gives $$|K|^r \geq C_{n\cdot
r} \cdot q^{n\cdot r - 1},$$ which translates into a bound of
$C_{n,r} \cdot q^{n - 1/r}$ on the size of $K$.
\end{proof}

We derive Theorem~\ref{thm-Kakeya} from a stronger theorem that
gives a bound on the size of sets that contain only `many' points
on `many' lines. Before stating the theorem we formally define
these sets.
\begin{define}[ $\bf{ (\delta,\gamma)}$-Kakeya
Set] A set $K \subset \F^n$ is a $(\delta,\gamma)$-{\sf Kakeya
Set} if there exists a set $\cL \subset \F^n$ of size at least
$\delta \cdot q^n$ such that for every $x \in \cL$ there is a line
in direction $x$ that intersects $K$ in at least $\gamma \cdot q$
points.
\end{define}
The next theorem, proven in Section~\ref{sec-Kakeya-Proof}, gives
a lower bound on the size of $(\delta,\gamma)$-Kakeya sets.
Theorem~\ref{thm-Kakeya} will follow by setting $\delta=\gamma=1$.
\begin{THM}\label{thm-Kakeya-Strong}
Let $K \subset \F^n$ be a $(\delta,\gamma)$-Kakeya Set. Then
\[ |K| \geq { d + n-1 \choose n-1 }, \]
where $$d = \left\lfloor q \cdot \min\{ \delta , \gamma \}
\right\rfloor - 2.$$
\end{THM}
Notice that, in order to get a bound of $\approx q^{n(1-\eps)}$ on
the size of $K$, Theorem~\ref{thm-Kakeya-Strong} allows $\delta$
and $\gamma$ to be as small as $q^{-\eps}$.

\subsection{Improving the bound to $\approx q^n$}
Following the initial publication of this work, Noga Alon and
Terence Tao \cite{AT08} observed that it is possible to turn the
proof of Theorem~\ref{thm-Kakeya} into a proof that gives a bound
of $C_n \cdot q^n$, thus achieving an optimal bound. We give below
a proof of this argument (the same argument gives an improvement
also for Theorem~\ref{thm-Kakeya-Strong}).
\begin{THM}
Let $K \subset \F^n$ be a Kakeya set. Then
\[ |K| \geq C_n \cdot q^{n}, \]
where $C_n$ depends only on $n$.
\end{THM}
\begin{proof}
Indeed, suppose this is false and let $K \subset F^n$ be a Kakeya
set of size less than ${{q+n-2} \choose {n}}$. Then there is a
nonzero polynomial of degree at most $q-1$ $P \in \F[x_1, \ldots
,x_n]$ so that $P(x)=0$ for all $x \in K$. Write
$P=\sum_{i=0}^{q-1}P_i$, where $P_i$ is the homogeneous part of
degree  $i$ of $P$. Fix $ y \in \F^n$. Then there is a $b \in
\F^n$ so that $P(b+ay)=0$ for all $a \in F$. For fixed $b$ and $y$
this is a polynomial of degree $q-1$ in $a$ which vanishes for all
$a \in \F$. It is thus identically zero, and hence all its
coefficients are zero. In particular, the coefficient of $a^{q-1}$
is zero, but it is easy to see that this is exactly $P_{q-1}(y)$.
Since $y$ was arbitrary it follows that the polynomial $P_{q-1}$
is identically zero. Therefore $P=\sum_{i=0}^{q-2} P_i$ and
repeating this argument we conclude that the polynomials
$P_{q-2},P_{q-3}, \ldots ,P_1$ are all identically zero. Hence $P$
is the constant term $P_0$, which has to be zero, as $P$ vanishes
at some points (including all points of $K$). This is a
contradiction, completing the proof.
\end{proof}

\section{Proof of Theorem~\ref{thm-Kakeya-Strong}}\label{sec-Kakeya-Proof}
We will use the following bound on the number of zeros of a degree
$d$ polynomial proven by Schwartz and Zippel
\cite{Schwartz80,Zippel79}.
\begin{lem}[Schwartz-Zippel]\label{lem-SZ}
Let $f \in \F[x_1,\ldots,x_n]$ be a non zero polynomial with
$\deg(f) \leq d$. Then
\[ |\{ x \in \F^n \,| f(x)=0 \}| \leq d \cdot q^{n-1}. \]
\end{lem}

\begin{proof}[Proof of Theorem~\ref{thm-Kakeya-Strong}]
Suppose in contradiction that $$|K| < { d + n-1 \choose n-1 }.$$
Then, the number of monomials in $\F[x_1,\ldots,x_n]$ of degree
$d$ is larger than the size of $K$. Therefore, there exists a
homogenous degree $d$ polynomial $g \in \F[x_1,\ldots,x_n]$ such
that $g$ is not the zero polynomial and
\[ \forall x \in K, \_ g(x)=0 \]
(this follows by solving a system of linear equations, one for
each point in $K$, where the unknowns are the coefficients of
$g$). Our plan is to show that $g$ has too many zeros and
therefore must be identically zero (which is a contradiction).

Consider the set
\[ K' \triangleq \{ c \cdot x \, | \, x \in K, c \in \F \} \]
containing all lines that pass through zero and intersect $K$ at
some point. Since $g$ is homogenous we have $$g(c \cdot x) = c^d
\cdot g(x)$$ and so
\[ \forall x \in K', \_ g(x)=0. \]

Since $K$ is a $(\delta,\gamma)$-Kakeya set, there exists a set
$\cL \subset \F^n$ of size at least $\delta \cdot q^n$ such that
for every $y \in \cL$ there exists a line with direction $y$ that
intersects $K$ in at least $\gamma \cdot q$ points.

\begin{claim}
For every $y \in \cL$ we have $g(y)=0$.
\end{claim}
\begin{proof}
Let $y \in \cL$ be some non zero vector (if $y=0$ then $g(y)=0$
since $g$ is homogenous). Then, there exists a point $z \in \F^n$
such that the line
\[ L_{z,y} = \{ z + a \cdot y | a \in \F \} \]
intersects $K$ in at least $\gamma \cdot q$ points. Therefore,
since $d+2 \leq \gamma \cdot q$, there exist $d+2$ distinct field
elements $a_1,\ldots,a_{d+2} \in \F$ such that $$\forall i \in
[d+2], \,\, z + a_i \cdot y \in K.$$ If there exists $i$ such that
$a_i=0$ we can remove this element from our set of $d+2$ points
and so we are left with at least $d+1$ distinct {\em non-zero}
field elements ( w.l.o.g $a_1,\ldots,a_{d+1}$) such that
\[\forall i \in
[d+1], \_ z + a_i \cdot y \in K \_ \text{and} \,\, a_i \neq 0\]
Let $b_i = a_i^{-1}$ where $i \in [d+1]$. The $d+1$ points
\[ w_i \triangleq b_i \cdot z + y , \,\, i \in [d+1] \]
are all in the set $K'$ and so
\[ g(w_i) = 0 , \,\, i \in [d+1]. \]
If $z=0$ then we have $w_i = y$ for all $i \in [d+1]$ and so
$g(y)=0$. We can thus assume that $z \neq 0$ which implies that
$w_1,\ldots,w_{d+1}$ are $d+1$ {\em distinct} points belonging to
the same line (the line through $y$ with direction $z$). The
restriction of $g(x)$ to this line is a degree $\leq d$ univariate
polynomial and so, since it has $d+1$ zeros (at the points $w_i$),
it must be zero on the entire line. We therefore get that $g(y)=0$
and so the claim is proven.
\end{proof}

We now get a contradiction since $$d/q < \delta$$ and, using
Lemma~\ref{lem-SZ}, a polynomial of degree $d$ can be zero on at
most a $d/q$ fraction of $\F^n$.
\end{proof}

\section{Acknowledgments}
I am  grateful  to Avi Wigderson for encouraging me to work on
this problem and for many helpful discussions. I thank my advisers
Ran Raz and Amir shpilka for their continuous support. I thank
Noga Alon, Richard Oberlin and Terrence Tao for pointing out the
improvements to Theorem~\ref{thm-Kakeya}.




\begin{thebibliography}{BKT04}

\bibitem[AT08]{AT08}
N.~Alon and T.~Tao.
\newblock Private communication.
\newblock 2008.

\bibitem[BKT04]{BKT04}
J.~Bourgain, N.~Katz, and T.~Tao.
\newblock A sum-product estimate in finite fields, and applications.
\newblock {\em GAFA}, 14(1):27--57, 2004.

\bibitem[Bou99]{Bou99}
J.~Bourgain.
\newblock On the dimension of {K}akeya sets and related maximal inequalities.
\newblock {\em Geom. Funct. Anal.}, (9):256--282, 1999.

\bibitem[Bou00]{Bou00}
J.~Bourgain.
\newblock Harmonic analysis and combinatorics: How much may they contribute to
  each other?
\newblock {\em IMU/Amer. Math. Soc.}, pages 13--32, 2000.

\bibitem[Dav71]{Davies71}
R.~Davies.
\newblock Some remarks on the {K}akeya problem.
\newblock {\em Proc. Cambridge Philos. Soc.}, (69):417--421, 1971.

\bibitem[KT99]{KT99}
N.~Katz and T.~Tao.
\newblock Bounds on arithmetic projections, and applications to the {K}akeya
  conjecture.
\newblock {\em Math. Res. Letters}, 6:625--630, 1999.

\bibitem[MT04]{MT04}
G.~Mockenhaupt and T.~Tao.
\newblock {Restriction and {K}akeya phenomena for finite fields}.
\newblock {\em Duke Math. J.}, 121:35--74, 2004.

\bibitem[Rog01]{Rogers01}
K.M Rogers.
\newblock The finite field {K}akeya problem.
\newblock {\em Amer. Math. Monthly 108}, (8):756--759, 2001.

\bibitem[Sch80]{Schwartz80}
J.~T. Schwartz.
\newblock Fast probabilistic algorithms for verification of polynomial
  identities.
\newblock {\em J. ACM}, 27(4):701--717, 1980.

\bibitem[Tao01]{Tao01}
T.~Tao.
\newblock From rotating needles to stability of waves: emerging connections
  between combinatorics, analysis, and pde.
\newblock {\em Notices Amer. Math. Soc.}, 48(3):294--303, 2001.

\bibitem[Tao08]{Tao08}
T.~Tao.
\newblock A new bound for finite field besicovitch sets in four dimensions.
\newblock {\em Pacific J. Math (to appear)}, 2008.

\bibitem[Wol99]{Wolff99}
T.~Wolff.
\newblock Recent work connected with the {K}akeya problem.
\newblock {\em Prospects in mathematics (Princeton, NJ, 1996).}
\newblock pages 129--162, 1999.

\bibitem[Zip79]{Zippel79}
R.~Zippel.
\newblock Probabilistic algorithms for sparse polynomials.
\newblock In {\em Proceedings of the International Symposiumon on Symbolic and
  Algebraic Computation}, pages 216--226. Springer-Verlag, 1979.

\end{thebibliography}
\end{document}